\pdfoutput=1
\documentclass[reqno]{amsart}

\usepackage[utf8]{inputenc}
\usepackage[english]{babel}

\usepackage{braket}
\usepackage{amsmath}
\usepackage{amstext}
\usepackage{amssymb}
\usepackage{amsthm}
\usepackage{mathtools}
\usepackage{stmaryrd}
\usepackage{mathrsfs}
\usepackage{extarrows}
\usepackage{afterpage}
\usepackage{tabu}
\usepackage{array}
\usepackage{faktor}

\usepackage{microtype}

\usepackage{graphicx}
\usepackage{amsfonts}

\usepackage{units}
\usepackage{enumerate}
\usepackage{url}
\usepackage{multirow}
\usepackage{algorithm}
\usepackage{algpseudocode}

\usepackage{float}
\newfloat{algorithm}{t}{lop}

\usepackage{IEEEtrantools}

\usepackage{tikz}
\usetikzlibrary{matrix,arrows,calc,fit,positioning}

\usepackage{hyperref}
\usepackage[english]{cleveref} 

\usepackage[margin=1cm]{caption}

\allowdisplaybreaks[4]

\theoremstyle{definition}
\newtheorem{theorem}{Theorem}[section]
\newtheorem{definition}[theorem]{Definition}
\newtheorem{remark}[theorem]{Remark}
\newtheorem{example}[theorem]{Example}
\newtheorem{lemma}[theorem]{Lemma}
\newtheorem{proposition}[theorem]{Proposition}
\newtheorem{corollary}[theorem]{Corollary}

\newtheorem{question}[theorem]{Question}

\newcommand{\N}{\mathbb{N}}
\newcommand{\Z}{\mathbb{Z}}

\newcommand{\C}{\mathbb{C}}
\newcommand{\Q}{\mathbb{Q}}

\renewcommand{\P}{\mathcal{P}}

\newcommand{\A}{\mathcal{A}}

\newcommand{\B}{\mathcal{B}}
\newcommand{\MM}{\mathcal{M}}
\renewcommand{\H}{\mathcal{H}}

\newcommand{\donotbreak}[1]{{}$\kern-2\mathsurround${}
  \binoppenalty10000 \relpenalty10000 #1{}$\kern-2\mathsurround${}}

\DeclareMathOperator{\rk}{rk}
\DeclareMathOperator{\GL}{GL}
\DeclareMathOperator{\Id}{Id}
\DeclareMathOperator{\Tor}{Tor}
\newcommand{\M}{\mathsf{M}}

\newcommand{\<}{\langle}
\renewcommand{\>}{\rangle}

\DeclareMathOperator{\cl}{cl}

\DeclareMathOperator{\Aff}{Aff}
\DeclareMathOperator{\Stab}{Stab}

\DeclareMathOperator{\HNF}{HNF}
\DeclareMathOperator{\SHNF}{SHNF}

\DeclareMathOperator{\mmod}{\,mod\,}

\newcommand{\rvline}{\hspace*{-\arraycolsep}\vline\hspace*{-\arraycolsep}}

\newcommand{\blockmatrix}[4]{
	\begin{pmatrix}
		#1 & \rvline & #2 \\
		\hline
		#3 & \rvline & #4
	\end{pmatrix}
}

\title[Representations of torsion-free arithmetic matroids]{Representations of torsion-free\\ arithmetic matroids}

\author[R. Pagaria]{Roberto Pagaria}
\address{Roberto Pagaria}
\address{\textup{Scuola Normale Superiore (Pisa, Italy)}}
\email{roberto.pagaria@sns.it}

\address{Giovanni Paolini}
\author[G. Paolini]{Giovanni Paolini}
\address{\textup{University of Fribourg (Switzerland)}}
\email{giovanni.paolini@sns.it}

\begin{document}

\begin{abstract}
We study the representability problem for torsion-free arithmetic matroids.
By using a new operation called ``reduction'' and a ``signed Hermite normal form'', we provide and implement an algorithm to compute all the representations, up to equivalence.
As an application, we disprove two conjectures about the poset of layers and the independence poset of a toric arrangement.  
\end{abstract}

\maketitle

\setcounter{tocdepth}{1}
\tableofcontents

\section{Introduction}
Arithmetic matroids are a generalization of matroids, inspired by the combinatorics of finite lists of vectors in $\Z^r$.
Representations of arithmetic matroids come from many different contexts, such as:
arrangements of hypertori in an algebraic torus; vector partition functions; zonotopes \cite{deconcini2010topics, DAdderioMoci2013, BrandenMoci2014}.
However, not all arithmetic matroids admit a representation.
A natural question is to determine whether a given arithmetic matroid is representable, and characterize all possible representations.
In this work, we give such a characterization in the case of torsion-free arithmetic matroids (i.e.\ when the multiplicity of the empty set is one).
Our characterization is effective, and it yields an explicit algorithm to compute all representations.
We implemented this algorithm as part of a new Sage library to work with arithmetic matroids, called Arithmat \cite{arithmat}.

After recalling some definitions (\Cref{sec:preliminaries}), we introduce two concepts that are used later: the \emph{strong gcd property} (\Cref{sec:strong-gcd}), and a new operation on \mbox{(quasi-)arithmetic} matroids that we call \emph{reduction} (\Cref{sec:reduction}).
Roughly speaking, the strong gcd property requires the multiplicity function to be uniquely determined by the multiplicity of the bases.
The idea behind the reduction operation is the following: given a central toric arrangement, we can quotient the ambient torus by the subgroup of translations globally fixing the arrangement.
In the quotient, the equations of the initial arrangement describe a new toric arrangement.
The initial arrangement defines an arithmetic matroid and the final one defines its reduction.
Indeed, the reduction operation consistently changes the multiplicity function, so that the resulting \mbox{(quasi-)arithmetic} matroid is torsion-free and surjective (i.e.\ the multiplicity of the empty set and of the full groundset are both equal to one).

In \Cref{sec:representations} we dive into the representation problem for torsion-free arithmetic matroids, which is the heart of our work.
We start by considering the surjective case: following ideas of \cite{Lenz2017,Pagaria2017}, we show that there is at most one representation, and we describe how to compute it.
Then we turn to the general case of torsion-free arithmetic matroids.
Here the representation need not be unique, and we describe how to compute all representations.
A consequence of our algorithm is that a torsion-free arithmetic matroid $(E, \rk, m)$ of rank $r$ has at most $m(E)^{r-1}$ essential representations up to equivalence.

The problem of recognizing equivalent representations reduces to the computation of a normal form of integer matrices up to left-multiplication by invertible matrices and change of sign of the columns.
We tackle this problem in \Cref{sec:shnf}, where we describe a polynomial-time algorithm to compute such a normal form.
We call this the \emph{signed Hermite normal form}, by analogy with the classical Hermite normal form (which is a normal form up to left-multiplication by invertible matrices).
The signed Hermite normal form is also implemented in Arithmat \cite{arithmat}.

In \Cref{sec:decomposition} we tackle a related algorithmic problem, namely finding the decomposition of a represented arithmetic matroid as a sum of indecomposable ones.

Finally, in \Cref{sec:applications} we describe a few applications of our software library Arithmat.
We disprove two known conjectures about the poset of layers and the arithmetic independence poset of a toric arrangement: 
we exhibit an arithmetic matroid with $13$ non-equivalent representations (i.e.\ central toric arrangements), whose associated posets are not Cohen-Macaulay, and therefore not shellable.
As already noted in \cite{Pagaria2018}, the toric arrangements associated with a fixed arithmetic matroid can have different posets of layers (in the previous example, the $13$ toric arrangements give rise to $3$ different posets of layers).
We conclude with the following open question: is the arithmetic independence poset of a toric arrangement uniquely determined by the associated arithmetic matroid?

\subsection*{Acknowledgments}
We thank Alessio d'Alì, Emanuele Delucchi, and Ivan Martino for the useful discussions.
This work was supported by the Swiss National Science
Foundation Professorship grant PP00P2\_179110/1.

\section{Preliminaries}
\label{sec:preliminaries}

In this section, we recall the basic definitions and properties of arithmetic matroids.
The main references are \cite{OxleyBook, DAdderioMoci2013, BrandenMoci2014}.
We define a matroid in terms of its rank function.

\begin{definition}
  A \emph{matroid} is a pair $\mathcal M = (E, \rk)$, where $E$ is a finite set and $\rk \colon \P(E) \to \N$ is a function satisfying the following properties:
  \begin{enumerate}
    \item $\rk(X) \leq \lvert X \rvert$ for every $X \subseteq E$;
    \item $\rk(X\cup Y) + \rk(X \cap Y) \leq \rk(X) + \rk(Y)$ for every $X,Y \subseteq E$;
    \item $\rk(X) \leq \rk(X \cup \{e\}) \leq \rk(X)+1$ for every $X \subseteq E$ and $e \in E$.
  \end{enumerate}
\end{definition}

For every matroid $\mathcal M=(E,\rk)$ and for every subset $X \subseteq E$, we denote by $\mathcal M/X$ the \textit{contraction} of $X$ and by $\mathcal M \setminus X$ the \textit{deletion} of $X$ (see \cite[Section 1.3]{OxleyBook}).
For $X\subseteq E$, denote by $\cl(X)$ the maximal subset $Y\supseteq X$ of rank equal to $\rk(X)$.

Let us recall the definition of arithmetic matroids, introduced in \cite{DAdderioMoci2013,BrandenMoci2014}.

\begin{definition}
  \label{def:molecule}
  A \textit{molecule} $(X,Y)$ of a matroid $\mathcal M$ is a pair of sets $X\subset Y \subseteq E$ such that the matroid $(\mathcal{M}/X)\setminus Y^c$ has a unique basis.
  Equivalently, it is possible to write $Y = X \sqcup T \sqcup F$ in such a way that, for every $Z$ with $X \subseteq Z \subseteq Y$, we have $\rk(Z) = \rk(X) + |Z \cap F|$.
  Here $F = Y \setminus \cl(X)$ is the unique basis of $(M/X) \setminus Y^c$, and $T = \cl(X) \cap Y \setminus X$ is the set of loops of $(M/X) \setminus Y^c$.
\end{definition}

\begin{definition}
  \label{def:arithmetic-matroid}
  An \textit{arithmetic matroid} is a triple $M = (E,\rk,m)$, such that $(E,\rk)$ is a matroid and $m \colon \P(E) \to \N_+=\set{ 1,2,\dots}$ is a function satisfying:
  \begin{enumerate}
  \item[(A1)] for every $X \subseteq E$ and $e \in E$, if $\rk(X\cup \{e\}) = \rk(X)$ then $m(X \cup \{e\}) \mid m(X)$, otherwise $m(X) \mid m(X \cup \{e\})$;
  
  \smallskip
  \item[(A2)] \label{eq:cond_on_molecule} if $(X,Y)$ is a molecule, with $Y = X \sqcup T \sqcup F$ as in \Cref{def:molecule}, then
  \[ m(X) \, m(Y) = m(X \cup T) \, m(X \cup F); \]
  
  \item[(P)] if $(X,Y)$ is a molecule, then
  \[ \sum_{X \subseteq S \subseteq Y} (-1)^{\left|X \cup F \right|-|S|} \, m(S) \geq 0. \]
  \end{enumerate}
  We call $m$ the \textit{multiplicity function}.
  If $M=(E,\rk,m)$ only satisfies axioms (A1) and (A2), we say that $M$ is a \emph{quasi-arithmetic matroid}.
  If $M$ satisfies only axiom (P), we say that it is a \emph{pseudo-arithmetic matroid}.
\end{definition}

If $m(\emptyset)=1$, we said that the arithmetic matroid $(E, \rk, m)$ is \textit{torsion-free}.
If $m(E)=1$, the matroid is \textit{surjective}.

Recall that any finitely generated abelian group $G$ has a (finite) torsion subgroup, which we denote by $\Tor(G)$, and a well-defined rank $\rk(G) \in \N$.

\begin{definition}
A \textit{representation} of an arithmetic matroid $M=(E,\rk,m)$ is a finitely generated abelian group $G$ together with elements $(v_e)_{e \in E}$ such that for all $X \subseteq E$ we have:
\begin{itemize}
\item $\rk (X)= \rk (\langle v_e\rangle_{e \in X})$,
\item $m(X)= \lvert \Tor(G/\langle v_e\rangle_{e \in X}) \rvert$,
\end{itemize}
where $\langle v_e\rangle_{e \in X}$ is the subgroup generated by $v_e$ for $e \in X$.

A representation is \textit{essential} if $\rk(G)=\rk(E)$.
\end{definition}

Notice that, if $M$ is torsion-free, every representation is a collection of integer vectors $(v_e)_{e \in E}$ in a lattice $\Lambda$, i.e.\ a free finitely generated abelian group.
Once a basis $\mathcal{B}$ of $\Lambda \simeq \Z^r$ is fixed, the vectors $(v_e)_{e \in E}$ can be identified with the columns of a matrix $A \in \M(r,|E|;\Z)$.
A different choice for the basis gives a matrix $A'=UA$ for some $U\in \GL(r;\Z)$.

\begin{definition}
Let $(G, (v_e)_{e \in E})$ and $(H, (w_e)_{e \in E})$ be two representations of an arithmetic matroid $M$.
The two representations are \textit{equivalent} if there exists
a group isomorphism $\varphi \colon G \to H$ such that $\varphi(\< v_e \>)= \< w_e \>$ for all $e\in E$.
\end{definition}

Notice that $\varphi(v_e) \in \set{w_e, -w_e}$, hence $\varphi(\< v_e \>_{e \in X})= \< w_e \>_{e \in X}$ for all $X \subseteq E$.

A topological motivation for the previous definitions comes from the fact that a representation of a torsion-free arithmetic matroid is a central toric arrangement.

\begin{definition}[Central toric arrangement]
	A central toric arrangement is a finite collection $\A$
	of hypertori in a torus $T \cong (\C^*)^r$, for some $r > 0$.
\end{definition}

Two representations are equivalent if and only if they describe isomorphic toric arrangements.

\section{The strong gcd property}
\label{sec:strong-gcd}

In this section, we introduce and study the strong gcd property.
This is a variant of the gcd property, which was introduced in \cite[Section 3]{DAdderioMoci2013}.
The gcd property is satisfied by all representable torsion-free arithmetic matroid, see \cite[Remark 3.1]{DAdderioMoci2013}.
The strong version is satisfied by all representable, surjective, and torsion-free arithmetic matroid, see \Cref{cor:representable-strong-gcd} below.

\begin{definition}
An arithmetic matroid $M = (E, \rk, m)$ satisfies the \emph{gcd property} if, for every subset $X \subseteq E$,
  \[ m(X) = \gcd \set{m(I) \mid I \subseteq X, \, \left| I \right| = \rk(I) = \rk (X)}. \]
\end{definition}

\begin{definition}
  \label{def:strong-gcd}
  An arithmetic matroid $M = (E, \rk, m)$ satisfies the \emph{strong gcd property} if, for every subset $X \subseteq E$,
  \[ m(X) = \gcd \set{m(B) \mid \text{$B$ basis and } \left|B \cap X \right| = \rk(X)}. \]
\end{definition}

\begin{lemma}
  \label{lemma:strong-implies-gcd}
  Let $M$ be an arithmetic matroid.
  If $M$ satisfies the strong gcd property, then it also satisfies the gcd property.
\end{lemma}

\begin{proof}
  For every independent set $I \subseteq E$, we have that
  \[ m(I) = \gcd \set{ m(B) \mid \text{$B$ basis and } I \subseteq B}. \]
  Then, for a generic subset $X \subseteq E$,
  \begin{align*}
    m(X) &= \gcd \set{ m(B) \mid \text{$B$ basis and } \left|B \cap X \right| = \rk(X)} \\
    &= \gcd \big\{ \gcd \set{ m(B) \mid \text{$B$ basis and } B \cap X = I } \mid I \subseteq X \text{ and }  \\
    & \;\;\quad\qquad \left|I\right| = \rk(I) = \rk(X) \big\} \\
    &\smash{\stackrel{(*)}{=}} \gcd \big\{ \gcd \set{ m(B) \mid \text{$B$ basis and } I \subseteq B } \mid I \subseteq X \text{ and } \\
    & \;\;\quad\qquad \left|I\right| = \rk(I) = \rk(X) \big\} \\
    &= \gcd \set{ m(I) \mid I \subseteq X \text{ and } \left|I\right| = \rk(I) = \rk(X) }.
  \end{align*}
  The equality $(*)$ follows from $|I| = \rk(I) = \rk(X) \geq \rk(B \cap X) = |B \cap X|$.
\end{proof}

\begin{lemma}
  \label{lemma:strong-gcd-dual}
  Let $M$ be an arithmetic matorid.
  If $M$ satisfies the strong gcd property, then its dual $M^*$ also satisfies the strong gcd property.
\end{lemma}

\begin{proof}
  Let $M = (E, \rk, m)$ and $M^* = (E, \rk^*, m^*)$.
  For every subset $X \subseteq E$, we have
  \begin{align*}
    m^*(X^c) &= m(X) = \gcd \set{ m(B) \mid \text{$B$ basis of $M$ and } \left|B \cap X\right| = \rk(X)} \\
    &\smash{\stackrel{(*)}{=}} \gcd \set{ m^*(B^c) \mid \text{$B^c$ is a basis of $M^*$ and } \left|B^c\cap X^c\right| = \rk^*(X^c) }.
  \end{align*}
  The equality $(*)$ follows from $|B^c \cap X^c| = |(B \cup X)^c| = |E| - (|B| + |X| - |B \cap X|) = |X^c| - |B| + |B \cap X| = |X^c| - \rk(E) + \rk(X) = \rk^*(X^c)$.
\end{proof}

\begin{theorem}
  \label{thm:strong-iff-gcd}
  Let $M$ be an arithmetic matroid.
  Then $M$ satisfies the strong gcd property if and only if both $M$ and $M^*$ satisfy the gcd property.
\end{theorem}

\begin{proof}
  If $M$ satisfies the strong gcd property, then the same is true for $M^*$ by Lemma \ref{lemma:strong-gcd-dual}, and therefore both $M$ and $M^*$ satisfy the gcd property by Lemma \ref{lemma:strong-implies-gcd}.
  
  Conversely, suppose that $M$ and $M^*$ both satisfy the gcd property.
  By the gcd property for $M$, for every $X \subseteq E$, we have
  \begin{equation}
    \label{eq:M-gcd}
    m(X) = \gcd \set{ m(I) \mid I \subseteq X \text{ and }  \left|I\right| = \rk(I) = \rk(X) }.
  \end{equation}
  By the gcd property for $M^*$, for every independent set $I \subseteq E$ we have
  \begin{align*}
    m(I) &= m^*(I^c) = \gcd \set{m^*(B^c) \mid B^c \subseteq I^c \text{ and } \left|B^c\right| = \rk^*(B^c) = \rk^*(I^c) } \\
    &= \gcd \set{m(B) \mid I \subseteq B \text{ and } \left|B^c\right| = \rk^*(B^c) = \rk^*(I^c) }.
  \end{align*}
  The condition $\left|B^c\right| = \rk^*(B^c) = \rk^*(I^c)$ can be rewritten as $\left|B^c\right| = |B^c| - \rk(E) + \rk(B) = |I^c| - \rk(E) + \rk(I)$.
  The first equality implies that $\rk(B) = \rk(E)$.
  By the second equality, we obtain $|B^c| = |I^c| - \rk(E) + |I| = |E| - \rk(E)$, thus $|B| = \rk(E)$.
  Therefore $B$ is a basis.
  Then
  \begin{equation}
    \label{eq:Mdual-gcd}
    m(I) = \gcd \set{m(B) \mid I \subseteq B \text{ and $B$ is a basis}}.
  \end{equation}
  In particular, if $I \subseteq X \subseteq E$ and $|I| = \rk(I) = \rk(X)$, then $\rk(I) \leq \rk(B \cap X) \leq \rk(X)$ and therefore $|B \cap X| = \rk(B \cap X) = \rk(X)$.
  Putting together \cref{eq:M-gcd,eq:Mdual-gcd}, we finally obtain
  \[ m(X) = \gcd \set{m(B) \mid \text{$B$ basis and } \left|B \cap X \right| = \rk(X)}. \]
  This proves the strong gcd property for $M$.
\end{proof}

\begin{corollary}
  \label{cor:representable-strong-gcd}
  Let $M$ be a surjective, torsion-free, and representable arithmetic matroid.
  Then $M$ satisfies the strong gcd property.
\end{corollary}

\begin{proof}
  By \cite[Remark 3.1]{DAdderioMoci2013}, a torsion-free representable arithmetic matroid satisfies the gcd property.
  In particular, this applies to $M$.
  Since $M$ is surjective and representable, its dual $M^* = (E, \rk^*, m^*)$ is torsion-free and representable, and thus it also satisfies the gcd property.
  By Theorem \ref{thm:strong-iff-gcd}, we deduce that $M$ satisfies the strong gcd property.
\end{proof}

As a final remark, notice that the strong gcd property is not preserved under deletion or contraction.

\section{Reduction of quasi-arithmetic matroids}
\label{sec:reduction}

In this section, we introduce a new operation on quasi-arithmetic matroids, which we call \emph{reduction}.
We will use this construction in the algorithm that computes the representations of a torsion-free arithmetic matroid.

\begin{definition}[Reduction]
  Let $M = (E, \rk, m)$ be a quasi-arithmetic matroid.
  Its reduction is the quasi-arithmetic matroid $\overline M = (E, \rk, \overline m)$ on the same groundset, with the same rank function, and with multiplicity function $\overline m$ is given by
  \[ \overline m(X) = \frac{\gcd \set{ m(B) \mid \text{$B$ is a basis, and } \rk(X) = \left|X \cap B\right| }}{\gcd \set{ m(B) \mid \text{$B$ is a basis} }}. \]
\end{definition}

Given a matroid $\MM = (E, \rk)$ and two subsets $X,Y \subseteq E$, define
\[ \begin{split}
  \B_{(X,Y)} = \{\, (B_1, B_2) \mid \; & \text{$B_1$ and $B_2$ are bases of $\MM$, } \rk(X) = |X \cap B_1|, \\
  &  \text{ and } \rk(Y) = |Y \cap B_2| \, \}.
\end{split} \]

\begin{lemma}
  \label{lemma:bijection}
  Let $\MM = (E, \rk)$ be a matroid, and let $(X,Y)$ be a molecule with $Y = X \sqcup T \sqcup F$ as in \Cref{def:molecule}.
  Then there is a bijection $\varphi \colon \B_{(X,Y)} \to \B_{(X \sqcup T, \, X \sqcup F)}$ given by
  \[ \varphi(B_1, B_2) = \big( (B_1 \setminus X) \cup (B_2 \cap (X \cup T)), \, 
  (B_2 \setminus (X \cup T)) \cup (B_1 \cap X) \big). \]
\end{lemma}

\begin{proof}
  Notice that $F \subseteq B_2$, because $\rk(Y) = \rk(Y \cap B_2) = \rk(X) + |B_2 \cap F|$ (the first equality is by definition of $\B_{(X,Y)}$, and the second equality is by definition of molecule).
  
  We want to prove that $B_3 = (B_1 \setminus X) \cup (B_2 \cap (X \cup T))$ is a basis.
  The set $B_1 \setminus X$ is independent, and its rank (or cardinality) is equal to $|B_1| - |X \cap B_1| = \rk(E) - \rk(X)$ by definition of $\B_{(X,Y)}$.
  The set $B_2 \cap (X \cup T)$ is also independent, and since $F \subseteq B_2$ its rank (or cardinality) is equal to $|B_2 \cap Y| - |F| = \rk(X) + |F| - |F| = \rk(X)$.
  Therefore $|B_3| \leq \rk(E)$.
  Applying property (2) of the rank function to the pair $(B_3, X \cup T)$, we obtain
  \[ \rk(B_3) + \rk(X \cup T) \geq \rk(B_3 \cup X \cup T) + \rk(B_3 \cap (X \cup T)). \]
  Notice that $\rk(X \cup T) = \rk(X)$ (by definition of molecule), $B_1 \subseteq B_3 \cup X \cup T$, and $B_2 \cap (X \cup T) \subseteq B_3 \cap (X \cup T)$.
  Then
  \[ \rk(B_3) + \rk(X) \geq \rk(B_1) + \rk(B_2 \cap (X \cup T)) = \rk(E) + \rk(X). \]
  Therefore $\rk(B_3) \geq \rk(E)$, and $B_3$ is a basis.
  
  We want now to check that $|B_3 \cap (X \cup T)| = \rk(X \cup T)$.
  We have $B_1 \cap T = \emptyset$, because
  \begin{align*}
    \rk(X) + |T \cap B_1| &= |X \cap B_1| + |T \cap B_1| = |(X \cap B_1) \sqcup (T \cap B_1)| \\
    &= |(X \cup T) \cap B_1| = \rk((X \cup T) \cap B_1) \\
    &\leq \rk(X \cup T) = \rk(X).
  \end{align*}
  Thus $B_3 \cap (X \cup T) = B_2 \cap (X \cup T)$, and this set has cardinality $\rk(X) = \rk(X \cup T)$.
  
  Similarly, $B_4 = (B_2 \setminus (X \cup T)) \cup (B_1 \cap X)$ is a basis, and $B_4 \cap (X \cup F)| = \rk(X \cup F)$.
  Therefore the map $\varphi$ is well-defined.
  
  The map $\psi \colon \B_{(X \sqcup T, \, X \sqcup F)} \to \B_{(X,Y)}$ defined by
  \[ \psi(B_3, B_4) = \big( (B_3 \setminus (X \cup T)) \cup (B_4 \cap X), \, (B_4 \setminus X) \cup (B_3 \cap (X \cup T)) \big) \]
  can be verified to be the inverse of $\varphi$.
  Therefore $\varphi$ is a bijection.
\end{proof}

\begin{lemma}
  \label{lemma:product}
  Let $M = (E, \rk, m)$ be a quasi-arithmetic matroid, and let $(X,Y)$ be a molecule with $Y = X \sqcup T \sqcup F$ as in \Cref{def:molecule}.
  If $\varphi \colon \B_{(X,Y)} \to \B_{(X \sqcup T, \, X \sqcup F)}$ is the bijection of \Cref{lemma:bijection}, and $(B_3, B_4) = \varphi(B_1, B_2)$, then
  \[ m(B_1) \, m(B_2) = m(B_3) \, m(B_4). \]
\end{lemma}

\begin{proof}
  Consider the following four molecules:
  \begin{align*}
    & (B_1 \cap X, \, (B_2 \cap (X \cup T)) \cup B_1); \\
    & (B_2 \cap (X \cup T), \, (B_1 \cap X) \cup B_2); \\
    & (B_2 \cap (X \cup T), \, (B_2 \cap (X \cup T)) \cup B_1); \\
    & (B_1 \cap X,\, (B_1 \cap X) \cup B_2).
  \end{align*}
  Applying axiom (A2) to these molecules, we get the following relations (we use the fact that $B_1 \cap T = \emptyset$, shown in the proof of \Cref{lemma:bijection}):
  \begin{align}
    & m(B_1 \cap X) \, m((B_2 \cap (X \cup T)) \cup B_1) = m((B_1 \cup B_2) \cap (X\cup T)) \, m(B_1); \\
    & m(B_2 \cap (X \cup T)) \, m((B_1 \cap X) \cup B_2) = m((B_1 \cup B_2) \cap (X\cup T)) \, m(B_2); \\
    & m(B_2 \cap (X \cup T)) \, m((B_2 \cap (X \cup T)) \cup B_1) = m((B_1 \cup B_2) \cap (X\cup T)) \, m(B_3); \\
    & m(B_1 \cap X) \, m((B_1 \cap X) \cup B_2) = m((B_1 \cup B_2) \cap (X\cup T)) \, m(B_4).
  \end{align}
  Let $k = m((B_1 \cup B_2) \cap (X\cup T))$.
  Multiplying the previous equations in pairs, we obtain $k^2 \, m(B_1) \, m(B_2) = k^2 \, m(B_3) \, m(B_4)$. Hence $m(B_1) \, m(B_2) = m(B_3) \, m(B_4)$.
\end{proof}

\begin{theorem}
  The reduction $\overline M$ of a quasi-arithmetic matroid $M = (E, \rk, m)$ is a torsion-free surjective quasi-arithmetic matroid, and it satisfies the strong gcd property.
\end{theorem}

\begin{proof}
  Let $d = \gcd \set{ m(B) \mid \text{$B$ is a basis} }$.
  We start by checking axiom (A1) of \Cref{def:arithmetic-matroid}.
  Consider a subset $X\subseteq E$ and an element $e \in E$.
  \begin{itemize}
    \item If $\rk(X \cup \{e\}) = \rk(X)$, then a basis $B$ such that $\rk(X) = \rk(X \cap B)$ also satisfies $\rk(X \cup \{e\}) = \rk((X \cup \{e\}) \cap B)$.
    Therefore $d \cdot \overline m(X \cup \{e\}) \mid d \cdot \overline m(X)$.
  
    \smallskip
    \item Similarly, if $\rk(X \cup \{e\}) = \rk(X) + 1$, then a basis $B$ such that $\rk(X \cup \{e\}) = \rk((X \cup \{e\}) \cap B)$ also satisfies $\rk(X) = \rk(X \cap B)$.
    Therefore $d \cdot \overline m(X) \mid d \cdot \overline m(X \cup \{e\})$.
  \end{itemize}

  We now check axiom (A2).
  Let $(X,Y)$ be a molecule, with $Y = X \sqcup T \sqcup F$ as in \Cref{def:molecule}.
  By definition of $\overline m$, we have that
  \[ d^2 \, \overline m(X) \, \overline m(Y) = \gcd \set{ m(B_1) \, m(B_2) \mid (B_1, B_2) \in \B_{(X,Y)} }. \]
  Similarly,
  \[ d^2 \, \overline m(X \cup T) \, \overline m(X \cup F) = \gcd \set{ m(B_3) \, m(B_4) \mid (B_3, B_4) \in \B_{(X \cup T, \, X \cup F)} }. \]
  By \Cref{lemma:bijection,lemma:product}, we obtain $d^2 \, \overline m(X) \, \overline m(Y) = d^2 \, \overline m(X \cup T) \, \overline m(X \cup F)$, hence $\overline m(X) \, \overline m(Y) = \overline m(X \cup T) \, \overline m(X \cup F)$.
  Therefore $\overline M$ is a quasi-arithmetic matroid.
  
  By definition of $\overline m$, we also have that $\overline m(\emptyset) = \overline m(E) = 1$, i.e.\ $M$ is torsion-free and surjective.
  It is also immediate to check that $\overline M$ satisfies the strong gcd property.
\end{proof}

It is not true in general that the reduction of an arithmetic matroid is an arithmetic matroid.
We see this in the following example.

\begin{example}
  Let $\MM = (E,\rk)$ be the uniform matroid of rank $2$ on the groundset $E=\{1,2,\dots,6\}$.
  Consider the multiplicity function $m \colon \P(E) \to \N_+$ defined as
  \begin{align*}
    & m(\emptyset)=1, & & \\
    & m(\{1\})=m(\{2\})=2, & & \\
    & m(\{ j \})=1 & & \text{if } j>2, \\
    & m(\{ X\})=1 & & \text{if } |X \cap \{3, \ldots, 6\}|\geq 2, \\
    & m(\{ i,j\})=2 & &\text{if } i=1,2 \text{ and } j>2, \\
    & m(\{ 1,2 \})=4, & & \\
    & m(\{ 1,2,3\})=1, & & \\
    & m(\{ 1,2,j\})=2 & & \text{if } j>3.
  \end{align*}
  Then $M = (E, \rk, m)$ is an arithmetic matroid (this can be checked using the software library Arithmat \cite{arithmat}).
  We have that $\overline m(X) = m(X)$ for every $X \subseteq E$, except that $\overline m(1,2,3) = 2$.
  The quasi-arithmetic matroid $\overline M = (E, \rk, \overline m)$ does not satisfy axiom (P) for the molecule $(\{1,2\}, E)$.
\end{example}

However, the reduction of a representable arithmetic matroid turns out to be a representable arithmetic matroid.

\begin{theorem}
  \label{thm:reduction-rep}
  If $M = (E, \rk, m)$ is a representable arithmetic matroid, then its reduction $\overline M$ is also a representable arithmetic matroid.
\end{theorem}

\begin{proof}
  Let $(v_e)_{e\in E} \subseteq G$ be a representation of $M$.
  Denote by $K$ the quotient of $G$ by its torsion subgroup $T$.
  Let $\overline G$ be the sublattice of $K$ generated by $\set{ \bar v_e \mid e \in E }$, where $\bar v_e$ is the class of $v_e$ in $K$.
  We are going to show that $(\bar v_e)_{e \in E} \subseteq \overline G$ is a representation of $\overline M$.
  
  Let $M' = (E, \rk, m')$ be the arithmetic matroid associated with the representation $(\bar v_e)_{e \in E} \subseteq \overline G$.
  By construction, $M'$ is representable, torsion-free (because $\overline G$ is torsion-free), and surjective (because the vectors $\bar v_e$ generate $\overline G$).
  Therefore, by \Cref{cor:representable-strong-gcd}, it satisfies the strong gcd property.
  As a consequence,
  \[ \gcd \set{ m'(B) \mid \text{$B$ basis} } = m(E) = 1. \]
  
  Let $B$ be a basis of $M$.
  Since $B$ is independent, we have that $T \cap \< v_b \>_{b \in B} = \set{0}$.
  Then,
  \begin{align*}
    m(B) &= \left| \faktor{G}{\< v_b \>_{b \in B}} \right| = |T| \cdot \left| \faktor{K}{\< \bar v_b \>_{b \in B}} \right| = |T| \cdot \left| \faktor{K}{\overline G} \right| \cdot \left| \faktor{\overline G}{\< \bar v_b \>_{b \in B}} \right| \\
    &= |T| \cdot \left| \faktor{K}{\overline G} \right| \cdot m'(B).
  \end{align*}
  If $B$ varies among all bases of $M$, taking the gcd of both sides we get
  \[ \gcd \set{ m(B) \mid \text{$B$ basis} } = |T| \cdot \left| \faktor{K}{\overline G} \right|. \]
  Therefore
  \[ m'(B) = \frac{m(B)}{ \gcd \set{ m(B) \mid \text{$B$ basis} } } = \overline m(B). \]
  Since both $M'$ and $\overline M$ satisfy the strong gcd property, $m'(X) = \overline m(X)$ for every subset $X \subseteq E$.
  This means that $\overline M = M'$ is representable.
\end{proof}

Finally, notice that the reduction does not commute with deletion and contraction.
However, it commutes with taking the dual.

\section{Representations of arithmetic matroids}
\label{sec:representations}

In this section, we prove that a torsion-free arithmetic matroid $M = (E, \rk, m)$ of rank $r$ has at most $m(E)^{r-1}$ essential representations, up to equivalence.
At the same time, we describe an algorithm to list all such essential representations.

Callegaro and Delucchi showed that matroids with a unimodular basis admit at most one representation \cite{CallegaroDelucchi2017}.
This result was later generalized by Lenz, in the case of weakly multiplicative matroids \cite{Lenz2017}.
The first author proved the uniqueness of the representation for surjective matroids and showed that general torsion-free matroids admit at most $m(E)^r$ essential representations \cite{Pagaria2017}.
In this work, we describe how to explicitly construct all representations, and improve the upper bound.

\subsection{Representation of torsion-free surjective matroids}
Consider a torsion-free surjective arithmetic matroid $M = (E, \rk, m)$ of rank $r$.
We want to describe how to choose $n = \vert E \vert$ vectors $(v_e)_{e \in E}$ in $\Z^r$ that form a representation of $M$ in the lattice $\Lambda = \< v_e \mid e \in E \>_\Z$, if $M$ is representable.

Let $B \subseteq E$ be a basis of $M$.
Relabel the groundset $E$ so that $E = \set{1, 2, \dots, n}$ and $B = \set{1, 2, \dots, r}$.
For $i=1,\dots, r$, define $v_i = m(B) \, e_i$ where $(e_1, \dots, e_r)$ is the canonical basis of $\Z^r$.
The absolute values of the coordinates of $v_{r+1}, \dots, v_n$ are uniquely determined by $M$, as described in \cite{Pagaria2017}.
The entries $a_{ij}$ of the matrix $A \in \M(r,n;\Z)$ with columns $v_1, \dots, v_n$ satisfy
\[ |a_{ij}| =
  \begin{cases}
    m(B \setminus \{i\} \cup \{j\}) & \text{if $B \setminus \{i\} \cup \{j\}$ is a basis}; \\
    0 & \text{otherwise}.
  \end{cases}
\]

To determine the signs of the entries $a_{ij}$, we follow the idea of Lenz \cite{Lenz2017}.
Consider the bipartite graph $G$ on the vertex set $E  = B \sqcup (E \setminus B)$, having an edge $(i,j)$ whenever $i \in B$, $j \in E \setminus B$, and $B \setminus \{i\} \cup \{j\}$ is a basis.
Let $F$ be a spanning forest of $G$.
Since reversing the sign of some vectors does not change the equivalence class of a representation, we can set $a_{ij}$ to be positive for $(i,j) \in F$ as shown by Lenz \cite[Lemma~6]{Lenz2017}.
We determine the signs of the remaining entries $a_{ij}$ by iterating the following procedure.
\begin{enumerate}
  \item Let $(i,j)$ be an edge of $G \setminus F$ such that the distance between $i$ and $j$ in $F$ is minimal.
  
  \item Let $i_1, j_1, i_2, j_2, \dots, i_k, j_k$ be a minimal path from $i$ to $j$ in $F$, where $i_1=i$ and $j_k=j$.
  Consider the $k\times k$ minor $A'$ of the matrix $A$ indexed by the rows $i_1,\dots, i_k$ and the columns $j_1, \dots, j_k$.
  Notice that the signs of all entries of $A'$ have already been determined, except for $a_{ij}$.
  The absolute value of the determinant of $A'$ must be equal to
  \[ | \det A'| =
    \begin{cases}
      m(B)^{k-1} \cdot m(B') & \text{if $B'$ is a basis} \\
      0 & \text{otherwise}
    \end{cases}
  \]
  where $B' = B \setminus \set{i_1, \dots, i_k} \cup \set{j_1, \dots, j_k}$.
  By minimality of the distance between $i$ and $j$, the only non-zero entries of $A'$ are $a_{i_\ell \, j_\ell}$ and $a_{i_\ell \, j_{\ell-1}}$ for $\ell=1,\dots, k$ (where $j_0=j_k$).
  Then
  \[ |\det A'| = \left| \prod_{\ell=1}^k a_{i_\ell \, j_\ell} - (-1)^k \prod_{\ell=1}^k a_{i_\ell \, j_{\ell-1}} \right|. \]
  Comparing the two given expressions of $|\det A'|$, the sign of $a_{ij}$ can be uniquely determined.
  
  \item Add the edge $(i,j)$ to $F$.
\end{enumerate}

At some iteration of this procedure, the equation
\[ \left| \prod_{\ell=1}^k a_{i_\ell \, j_\ell} - (-1)^k \prod_{\ell=1}^k a_{i_\ell \, j_{\ell-1}} \right| \; = \; \begin{cases}
      m(B)^{k-1} \cdot m(B') & \text{if $B'$ is a basis} \\
      0 & \text{otherwise}
    \end{cases} \]
of the second step might have no solution.
If this happens, we can conclude that the matroid $M$ is not representable.

\begin{remark}
  \label{rmk:orientability}
  If $M$ is orientable (in the sense of \cite{PagariaOAM}), then there exists a \emph{chirotope} $\chi\colon E^r \to \set{-1, 0, 1}$ such that $a_{ij} = \chi(B \setminus \{i\} \cup \{j\}) \cdot m(B \setminus \{i\} \cup \{j\})$.
  This ensures that the equation of step (2) always has a solution.
  Conversely, a failure of step (2) implies that $M$ is not orientable.
\end{remark}

We have finally constructed a matrix $A$ whose columns $(v_e)_{e \in E}$ form a candidate representation of $M$ in the lattice $\Lambda = \< v_e \mid e \in E \>_\Z$.
To recover the coordinates of the vectors $(v_e)_{e \in E}$ with respect to a basis of $\Lambda$, we use the Smith normal form as explained by the following lemma.

\begin{lemma}
  Let $(v_e)_{e \in E}$ be a set of vectors in $\Z^r$, with coordinates described by a matrix $A \in \M(r,n;\Z)$ of rank $r$.
  Let $D = UAV$ be the Smith normal form of $A$.
  Then the $r \times n$ matrix consisting of the first $r$ rows of $V^{-1}$ gives the coordinates of the vectors $(v_e)_{e \in E}$ with respect to a basis of $\Lambda = \< v_e \mid e \in E \>_\Z$.
\end{lemma}

\begin{proof}
  Recall that $U \in \GL(r;\Z)$, $V \in \GL(n;\Z)$, and $D \in \M(r, n; \Z)$.
  Let $D = D'I$, where $I$ is the block matrix $(\Id_{r \times r} \mid 0) \in \M(r,n;\Z)$ and $D' \in \M(r,r;\Z)$ is the matrix consisting of the first $r$ columns of $D$.
  Consider the vectors $w_1, \dots, w_r$ of $\Z^r$ given by the columns of $U^{-1}D'$.
  Then $U^{-1}D' \cdot IV^{-1} = A$, and therefore the columns of $IV^{-1}$ are the coordinates of $(v_e)_{e \in E}$ with respect to the $\Q$-basis $\B = (w_1, \dots, w_r)$.
  Since the matrix $IV^{-1}$ has integer entries and Smith normal form equal to $I$, the basis $\B$ is also a lattice basis of $\Lambda$.
  We conclude by noticing that $IV^{-1}$ is the matrix consisting of the first $r$ rows of $V^{-1}$.
\end{proof}

At this point, we have a candidate representation of the matroid $M$.
If $M$ is representable, this is the only possible representation of $M$ up to equivalence.
We only need to verify if it is indeed a representation of $M$, checking the multiplicity $m(X)$ for every subset $X \subseteq E$.

\begin{remark}
  Under our assumptions ($m(\emptyset) = m(E) = 1$), the matroid $M$ is representable if and only if it is orientable and satisfies the strong gcd property \cite[Proposition~8.3]{PagariaOAM}.
  Before the final check of our algorithm, $M$ is known to be orientable by Remark \ref{rmk:orientability}.
  Then the final check has a positive result if and only if $M$ satisfies the strong gcd property.
\end{remark}

\subsection{Representations of general torsion-free matroids}

In this section, we describe how to construct all essential representations (up to equivalence) of a general torsion-free matroid $M = (E, \rk, m)$.
Let $r = \rk(E)$.

Consider the reduction $\overline M$.
If $M$ is representable, then $\overline M$ must also be a representable arithmetic matroid by \Cref{thm:reduction-rep}.
Since $\overline M$ is torsion-free and surjective, using the algorithm of the previous section we can check if $\overline M$ is a representable arithmetic matroid.
Assume from now on that this is the case. Then the previous algorithm also yields the unique essential representation of $\overline M$ (up to equivalence), which consists of some integer matrix $A \in \M(r,n;\Z)$.

\begin{theorem}
  \label{thm:representations}
  If $A \in \M(r,n;\Z)$ is an essential representation of $\overline M$, then every essential representation of $M$ is equivalent to $HA$ for some matrix $H \in \M(r,r;\Z)$ in Hermite normal form, with $\det(H) = m(E)$.
\end{theorem}

\begin{proof}
  Every essential representation $C \in \M(r,n;\Z)$ of $M$ induces an essential representation $A' \in \M(r,n;\Z)$ of $\overline M$, as shown in the proof of \Cref{thm:reduction-rep}.
  These two representations are related as follows: $C' = H'A'$, where the matrix $H' \in \M(r,r;\Z)$ describes (in the chosen coordinates) the inclusion $\overline G \hookrightarrow K$, and has rank $r$.
  Since all representations of $\overline M$ are equivalent, we can write $A' = U' A S$ for some integer matrices $U' \in \GL(r;\Z)$ and $S \in \Z_2^n \subseteq \GL(n,\Z)$.
  Then we have $CS = H'U'A$.
  Let $U \in \GL(r, \Z)$ be an integer matrix such that $UH'U'$ is in Hermite normal form.
  We obtain that the representation $UCS$ of $M$ is equivalent to $C$ and can be written as $UCS = HA$, where $H=UH'U'$ is in Hermite normal form.
  Notice that $m(E) = \det(H) \cdot \overline m(E) = \det(H)$.
\end{proof}

Some of the representations given by \Cref{thm:representations} can be equivalent.
To compute a list of representatives of the equivalence classes of representations, one needs to compute a normal form of matrices in $\M(r,n;\Z)$ up to left-multiplication by $\GL(r;\Z)$ and change of sign of the columns.
We develop an algorithm to do this in the next section.

A direct consequence of \Cref{thm:representations} is a new upper bound on the number of non-equivalent representations of a torsion-free matroid.

\begin{corollary}
	Every torsion-free arithmetic matroid $M = (E, \rk, m)$ of rank $r$ has at most $m(E)^{r-1}$ equivalence classes of essential representations.
\end{corollary}

\begin{proof}
	Reorder the groundset of $M$ so that the first $r$ elements form a basis.
	Let $A \in \M(r,n;\Z)$ be an essential representation of the reduction $\overline M = (E, \rk, \overline m)$.
	Without loss of generality, we can assume that $A$ is in Hermite normal form.
	
	Let $B_i$ be the $i\times i$ leading principal minor of a matrix $B$.
	For every $i \in \set{1, \dots, r}$, we have that $A_i$ is upper triangular and $\det(A_i) = \overline m(\set{1,\dots, i})$.
	If $H \in \M(r,r,\Z)$ is an upper triangular matrix such that $HA$ is a representation of $M$, then $\det(H_i) \det(A_i) = \det((HA)_i) = m(\set{1,\dots,i})$.
	By \Cref{thm:representations}, every essential representation of $M$ is equivalent to $HA$ for some matrix $H \in \M(r,r,\Z)$ in Hermite normal form such that $\det(H) = m(E)$.
	The diagonal entries $d_1, \dots, d_r$ of $H$ are uniquely determined by the previous relations.
	The number of such matrices $H$ is $\prod_{i=1}^r d_i^{i-1} \leq \prod_{i=1}^r d_i^{r-1} = m(E)^{r-1}$.	
\end{proof}

\begin{remark}
	The orientability of $M$ is equivalent to the orientability of the reduction $\overline M$.
	Then \Cref{rmk:orientability} yields an algorithm to check the orientability of $M$.
\end{remark}

\section{Signed Hermite normal form}
\label{sec:shnf}

In this section, we describe an algorithm that takes as input a matrix $A \in \M(r,n;\Z)$ and outputs a normal form with respect to the action of $\GL(r;\Z) \times \Z_2^n$.
Here $\GL(r;\Z)$ acts on $\M(r,n;\Z)$ by left-multiplication, and the $j$-th standard generator of $\Z_2^n$ acts by changing the sign of the $j$-th column.
It is convenient to view the elements of $\Z_2^n$ as the $n\times n$ diagonal matrices with diagonal entries equal to $\pm 1$.
Then a pair $(U, S) \in \GL(r;\Z) \times \Z_2^n$ acts on $\M(r,n;\Z)$ as $A \mapsto UAS$.

Recall that the (left) Hermite normal form is a canonical form for matrices in $\M(r,n;\Z)$ with respect to the left action of $\GL(r;\Z)$ (see for instance \cite{newman1972integral} and \cite{cohen1993course}).
We write $\HNF(A)$ for the Hermite normal form of $A$.
A matrix in Hermite normal form satisfies the following properties:
\begin{itemize}
	\item it is an upper triangular $r\times n$ matrix, and zero rows are located below non-zero rows;
	\item the pivot (i.e.\ the first non-zero entry) of a non-zero row is positive, and is strictly to the right of the pivot of the row above it;
	\item the elements below pivots are zero, and the elements above a pivot $q$ are non-negative and strictly smaller than $q$.
\end{itemize}
Our normal form with respect to the action of $\GL(r;\Z) \times \Z_2^n$ has a simple definition in terms of the Hermite normal form.
We call it the \emph{signed Hermite normal form}.

\begin{definition}
	The \emph{signed Hermite normal form} $\SHNF(A)$ of a matrix $A \in \M(r,n;\Z)$ is the lexicographically minimal matrix in the set $\set{ \HNF(AS) \mid S \in \Z_2^n }$.
	To compare two matrices lexicographically, we look at the columns from left to right, and in each column, we look at the entries from bottom to top.
\end{definition}

\begin{remark}
	By definition, a matrix in signed Hermite normal form is also in Hermite normal form.
\end{remark}

\begin{example}
	Consider the following sequence of $2\times 2$ matrices:
	\[
	\begin{pmatrix}
	4 & 2 \\
	0 & 3
	\end{pmatrix}
	\longrightarrow
	\begin{pmatrix}
	4 & -2 \\
	0 & -3
	\end{pmatrix}
	\longrightarrow
	\begin{pmatrix}
	4 & 1 \\
	0 & 3
	\end{pmatrix}.
	\]
	The leftmost matrix is in Hermite normal form, but it is not in signed Hermite normal form.
	Indeed, if we change the sign of the second column, we obtain the matrix in the middle; its Hermite normal form is given by the rightmost matrix, which is lexicographically smaller than the leftmost one.
\end{example}

A naive algorithm to compute the signed Hermite normal form could be: try all the $2^n$ elements $S \in \Z_2^n$; determine the left Hermite normal form of $AS$;  choose the lexicographically minimal result.
This algorithm runs in $2^n \cdot \text{poly}(n,r)$.
In the rest of this section, we are going to describe an algorithm which is polynomial in $n$ and $r$.

Given a matrix $A \in \M(r,n;\Z)$, we indicate by $A_j \in \Z^r$ the $j$-th column of $A$, and by $A_{:j} \in \M(r,j;\Z)$ the matrix consisting of the first $j$ columns of $A$.
We write $\Z_2^j$ for the subgroup of $\Z_2^n$ generated by the first $j$ standard generators of $\Z_2^n$.
Also, for every $m \leq r$, we regard the group $\GL(m;\Z)$ as a subgroup of $\GL(r;\Z)$ via the natural inclusion
\[ U \mapsto \blockmatrix{U}{0}{0}{I_{(r-m)\times (r-m)}}. \]
Define the stabilizer $\Stab(B)$ of a matrix $B \in \M(r,j;\Z)$ as the subgroup
\[ \Stab(B) = \set{ S \in \Z_2^{j} \mid BS = UB \text{ for some } U \in \GL(r;\Z) } \subseteq \Z_2^j. \]
This is the stabilizer of the orbit $\set{ U B \mid U \in \GL(r;\Z)}$ with respect to the right action of $\Z_2^j$.
Notice that $\Stab(B) = \Stab(VBT)$ for every $(V,T) \in \GL(r;\Z) \times \Z_2^{j}$, because $\Z_2^j$ is abelian.

The pseudocode to compute the signed Hermite normal form is given in \Cref{alg:s-normal-form}.
In the rest of this section, we are going to explain it with more details.

\begin{algorithm}
	\caption{Signed Hermite normal form}
	\label{alg:s-normal-form}
	
	\textbf{Input}: a matrix $A \in \M(r,n; \Z)$.\\
	\textbf{Output}: the signed Hermite normal form of $A$.
	
	\begin{algorithmic}[1]
		\State $G \gets \set{0}$, as a subgroup of $\Z_2^n$
		\State $A \gets$ Hermite normal form of $A$ \label{line:hnf-A}    
		
		\For{$j=1, 2, \dotsc, n$}\label{line:outer-loop}
		
		\State $m \gets \rk(A_{: j-1})$\label{line:rank}
		\State $q \gets A_{m+1, j}$\label{line:pivot}
		
		\State $\varphi \gets$ the group homomorphism $G \to \GL(m;\Z) \subseteq \GL(r;\Z)$ which maps $S \in G$ to the unique matrix $\varphi(S) \in \GL(m;\Z)$ such that $\varphi(S)A_{:j-1}S=A_{:j-1}$\label{line:phi}
		\State $G \gets G \times \Z_2$, where $\Z_2$ is the $j$-th factor of $\Z_2^n$\label{line:augment-G}
		\State Extend $\varphi$ to a group homomorphism $G \to \GL(m;\Z) \subseteq \GL(r;\Z)$, by sending the generator of the new $\Z_2$ factor to $-I_{m\times m}$\label{line:extend-phi}
		
		\For{$i=m, m-1,\dotsc, 1$}\label{line:inner-loop}
		\State $O \gets \set{ (\varphi(S) A_j)_{i} \!\mod q | S \in G }$\label{line:orbit}
		\State $u \gets \min O$\label{line:min-orbit}
		\State $\overline S \gets$ any element of $G$ such that $(\varphi(\overline S) A_j)_i \!\mod q = u$\label{line:S-min-orbit}
		\State $A \gets$ Hermite normal form of $A \overline S$\label{line:HNF-inner-loop}
		\State $G \gets \set{S \in G | (\varphi(S)A_j)_i \mmod q = u }$\label{line:restrict-G}
		\State $\varphi \gets \varphi |_G$\label{line:restrict-phi}
		\EndFor\label{line:end-inner-loop}
		\EndFor
		\State \Return $A$
	\end{algorithmic}
\end{algorithm}

Throughout the execution of \Cref{alg:s-normal-form}, $G$ is always a subgroup of $\Z_2^n$.
It would require exponential time and space to compute and store the list of all its elements.
For this reason, we rather describe it by giving one of its $\Z_2$-bases, i.e.\ a list of $k$ linearly independent vectors in $\Z_2^n$ (where $k$ is the dimension of $G$ as a $\Z_2$-vector space).
Accordingly, the group homomorphism $\varphi \colon G \to \GL(r;\Z)$ is always described by giving its values on the $\Z_2$-basis of $G$.

\Cref{alg:s-normal-form} adjusts the columns one at a time, from left to right.
This is possible thanks to the following observation.

\begin{lemma}
	For every $j$ we have $\SHNF(A)_{:j} = \SHNF(A_{:j})$. In particular, the $j$-th column of $\SHNF(A)$ only depends on the first $j$ columns of $A$.
\end{lemma}

\begin{proof}
	It is well known that $\HNF(A)_{:j} = \HNF(A_{:j})$.
	Therefore
	\begin{align*}
	\SHNF(A)_{:j} &= \min \set{\HNF(AS) \mid S \in \Z_2^n}_{:j} \\
	&= \min \set{\HNF(AS)_{:j} \mid S \in \Z_2^n} \\
	&= \min \set{\HNF((AS)_{:j}) \mid S \in \Z_2^n} \\
	&= \min \set{\HNF(A_{:j}T) \mid T \in \Z_2^j} \\
	&= \SHNF(A_{:j}).
	\end{align*}
	In the second equality we used the fact that the lexicographic order privileges the first $j$ columns over the last $n-j$.
\end{proof}

Let $j$ be the current column (\cref{line:outer-loop}).
At the beginning of each iteration of the outer for loop, the following properties hold:
\begin{enumerate}[(i)]
	\item $A_{:j-1}$ is in signed Hermite normal form;
	\item $A_{:j}$ is in Hermite normal form;
	\item $G = \Stab(A_{:j-1}) \subseteq \Z_2^{j-1}$.
\end{enumerate}
The proof is by induction: for $j=1$ this is trivial; the induction step is given by \Cref{rmk:outer-loop}.
Notice that $\Stab(A_{:j-1})$ describes the freedom we still have in changing the sign of the first $j-1$ columns, without affecting the Hermite normal form of $A_{:j-1}$.

Let $m = \rk(A_{:j-1})$ (\cref{line:rank}) and $q = A_{m+1, j}$ (\cref{line:pivot}).
Here $q \geq 0$, and if $q \neq 0$ then $q$ is a pivot of the Hermite normal form $A_{:j}$.
The matrix $A_{:j}$ looks like this:
\begin{equation}\label{eq:Aj}
A_{:j} = 
\left(
\begin{array}{c|c}\\[-5pt]
\;\;B\;\;\; & v \\[5pt]
\hline
& q \\
\cline{2-2}
\multicolumn{1}{c}{\;\;0\;\;\;} & \multicolumn{1}{c}{} \\[5pt]
\end{array}
\right)
\end{equation}
where $B \in \M(m, j-1; \Z)$, and $v$ is a column vector in $\Z^m$.

In \cref{line:phi}, we consider the group homomorphism $\varphi \colon G \to \GL(m;\Z)$ defined as follows: $\varphi(S)$ is the unique matrix in $\GL(m;\Z) \subseteq \GL(r;\Z)$ such that $\varphi(S) A_{:j-1}S = A_{:j-1}$.
Since $A_{:j-1}$ has zeros in the last $r-m$ rows, this condition is equivalent to $\varphi(S)BS = B$.
As we said before, we describe $\varphi$ by computing the image of each element $S$ of the $\Z_2$-basis of $G$.
This is done by running the algorithm for the Hermite normal form of $BS$: this algorithm returns both the Hermite normal form (which we already know to be equal to $B$) and a matrix $\varphi(S) \in \GL(m;\Z)$ such that $\varphi(S) B S = B$.
The matrix $\varphi(S)$ is unique because $B$ has rank $m$.
Notice that $\varphi$ is indeed a group homomorphism, because
\begin{align*}
\varphi(ST)BST &= B \\
& = \varphi(S)BS \\
& = \varphi(S)\varphi(T)BTS \\
& = \varphi(S)\varphi(T)BST.
\end{align*}

In \cref{line:augment-G}, we replace $G$ with $G \times \Z_2$, where $\Z_2$ is the $j$-th factor of $\Z_2^n$.
This is achieved by extending the basis of $G$ with the $j$-th element of the standard $\Z_2$-basis of $\Z_2^n$.
In \cref{line:extend-phi} we extend $\varphi$ to the new basis element, by sending it to $-I_{m\times m}$.
The definition of $\varphi$ is motivated by \Cref{lemma:phi} below.

Given $x \in \Z$, define $x \mmod q$ as the remainder of the division between $x$ and $q$ if $q > 0$, and define $x \mmod 0 = x$.
Since the group $G$ is going to change throughout the inner loop, it is convenient to denote by $G_0$ the value of $G$ after \cref{line:extend-phi} is executed.

\begin{lemma}\label{lemma:phi}
	Write $A_{:j}$ as in \cref{eq:Aj}.
	Let $\H = \set{ \HNF(A_{:j}S) \mid S \in \Z_2^j} \subseteq \M(r, j, \Z)$, and let $\H' = \set{A' \in \H \mid A'_{:j-1} = A_{:j-1}}$.
	Then the matrices in $\H'$ are precisely those of the form
	\[
	A' =
	\left(
	\begin{array}{c|c}\\[-5pt]
	\;\;B\;\;\; & w \\[5pt]
	\hline
	& q \\
	\cline{2-2}
	\multicolumn{1}{c}{\;\;0\;\;\;} & \multicolumn{1}{c}{} \\[5pt]
	\end{array}
	\right)
	\]
	where $w = \varphi(S)\cdot v$ mod $q$ for some $S \in G_0$.
\end{lemma}

\begin{proof}
	Denote by $S_j$ the $j$-th element of the standard $\Z_2$-basis of $\Z_2^n$.
	Then $S_j$ acts on $\M(r,n;\Z)$ by changing the sign of the $j$-th column.
	Every element of $G_0$ is of the form $S_0S_j^\epsilon$ for some $S_0 \in \Stab (A_{:j-1})$ and $\epsilon \in \set{0,1}$.
	
	We first show that a matrix $A'$ as above (for a given $S \in G_0$) belongs to $\H'$.
	Notice that $A'$ is in Hermite normal form, and $A'_{:j-1} = A_{:j-1}$, so it is enough to show that $A'$ is in the same orbit as $A_{:j} S$ with respect to the left action of $\GL(r;\Z)$.
	Write $S = S_0S_j^\epsilon$.
	Then, by definition of $\varphi$, we have
	\[
	\varphi(S_0) A_{:j} S_0 S_j^\epsilon =
	\left(
	\begin{array}{c|c}\\[-5pt]
	\;\;B\;\;\; & \varphi(S_0) \cdot v \\[5pt]
	\hline
	& q \\
	\cline{2-2}
	\multicolumn{1}{c}{\;\;0\;\;\;} & \multicolumn{1}{c}{} \\[5pt]
	\end{array}
	\right) \cdot S_j^\epsilon
	=
	\left(
	\begin{array}{c|c}\\[-5pt]
	\;\;B\;\;\; & \varphi(S) \cdot v \\[5pt]
	\hline
	& (-1)^\epsilon q \\
	\cline{2-2}
	\multicolumn{1}{c}{\;\;0\;\;\;} & \multicolumn{1}{c}{} \\[5pt]
	\end{array}
	\right),
	\]
	and this matrix is in the same orbit as $A'$.
	
	Conversely, let $A' \in \H'$. Since $A' \in \H$, we have $A' = \HNF(A_{:j}S)$ for some $S \in \Z_2^j$.
	Write $S = S_0 S_j^\epsilon$ for some $S_0 \in \Z_2^{j-1}$ and $\epsilon \in \set{0,1}$.
	We have $A'_{:j-1} = \HNF(A_{:j}S)_{:j-1} = \HNF(A_{:j-1} S_0)$.
	Since $A' \in \H'$, we also have $A'_{:j-1} = A_{:j-1}$. Therefore $\HNF(A_{:j-1}S_0) = A_{:j-1}$, which implies that $S_0 \in \Stab (A_{:j-1})$.
	Then $S = S_0 S_j^\epsilon \in G_0$.
	We conclude by noticing that $A' = \HNF(A_{:j}S)$ is of the form given in the statement, with $w = \varphi(S)\cdot v \mmod q$.
\end{proof}

\Cref{lemma:phi} gives an explicit characterization of the possible values of the $j$-th column of the Hermite normal form of $AS$, provided that $\HNF(AS)_{:j-1} = A_{:j-1}$.
Then, to compute the $j$-th column of the signed Hermite normal form, we need to find the lexicographically minimal vector $w = \varphi(S) \cdot v$ mod $q$.
This is done in the inner loop, starting from the $m$-th row and going up to the first row (lines \ref{line:inner-loop}-\ref{line:end-inner-loop}).

Let $i$ be the current row (\cref{line:inner-loop}).
At the beginning of each iteration of the inner loop, we have that
\begin{equation}
G = \set{S \in G_0 \mid (\varphi(S)A_j)_{i'} \mmod q = A_{i',j} \text{ for all } i'>i }.
\label{eq:subgroup}
\end{equation}
This is proved by induction: it holds at the beginning of the first iteration ($i=m$) because $G=G_0$; the induction step is proved below.

In \cref{line:orbit} we explicitly compute the set $O$ of all possible values of the entry $(i,j)$.
For ease of notation, write $A_{:j}$ as in \cref{eq:Aj}.
Then
\begin{align*}
O &= \set{ (\varphi(S) A_j)_i \!\!\!\mod q \mid S \in G } \\
&= \set{ (\varphi(S) \cdot v)_i \!\!\!\mod q \mid S \in G }.
\end{align*}
A key observation is that the set $O$ is very small, even if $G$ can be large.

\begin{lemma}
	In \cref{line:orbit}, $\lvert O \rvert \in \set{1,2,4}$.
\end{lemma}

\begin{proof}
	By eq.\ \eqref{eq:subgroup} and since $\varphi(S)$ is upper triangular, 
	there is a well-defined action of $G$ on $\Z_q$: an element $S \in G$ acts as an affine automorphism $\rho(S) \in \Aff(\Z_q)$ given by
	\[ x \;\mapsto\; \varphi(S)_{i,i} \, x + \sum_{k=i+1}^{r} \varphi(S)_{i,k} A_{k,j}. \]
	This is how $G$ acts on the entry $(i,j)$ of $A$.
	Notice that $\rho(S)$ has the form $x \mapsto \pm x + \beta$ for some $\beta \in \Z_q$, since $\varphi(S)_{i,i} = \pm 1$.
	In addition, $\rho(S)$ is an involution, so it has one of the following forms:
	\[ x \mapsto x, \quad x \mapsto x + q/2 \,\text{ (if $q$ is even)}, \quad x \mapsto -x + \beta \text{ for some } \beta \in \Z_q. \]
	The maps $x \mapsto x$ and $x \mapsto x + q/2$ commute with each other and with any map of the form $x \mapsto -x+\beta$.
	However, given two maps of the form $x \mapsto -x + \beta_1$ and $x \mapsto -x+\beta_2$, they commute if and only if $2(\beta_1 - \beta_2) = 0$, i.e.\ $\beta_1 = \beta_2$ or $\beta_1 = \beta_2 + q/2$.
	Since $\rho(G)$ is abelian, there exists a $\beta \in \Z_q$ such that any element $S \in G$ acts as one of the following four maps:
	\[ x \mapsto x, \quad x \mapsto x + q/2, \quad x \mapsto -x + \beta, \quad x \mapsto -x + \beta + q/2. \]
	Therefore $\rho(G)$ is isomorphic to a subgroup of $\Z_2^2$.
	By definition, $O$ is the orbit of $A_{i,j}$ in $\Z_q$, and so its cardinality divides $4$.
\end{proof}

In \cref{line:min-orbit}, we select the smallest element $u \in O$.
In \cref{line:S-min-orbit}, we choose any element $\overline S \in G$ such that $\rho(\overline S)(A_{i,j}) = u$.
After \cref{line:HNF-inner-loop}, the entry $(i,j)$ of $A$ is equal to $u$.
Finally, in \cref{line:restrict-G} we update the group $G$ in order to satisfy eq.\ \eqref{eq:subgroup}, and in \cref{line:restrict-phi} we restrict $\varphi$ to the new group $G$.

\begin{remark}
	\label{rmk:outer-loop}
	At the end of the inner loop (after \cref{line:end-inner-loop}), we have that: $G = \Stab(A_{:j})$, by eq.\ \eqref{eq:subgroup} for $i=0$;
	$A$ is in Hermite normal form, by \cref{line:HNF-inner-loop}, so in particular $A_{:j+1}$ is in Hermite normal form;
	$A_{:j}$ is in signed Hermite normal form, by construction.
\end{remark}

\begin{example}
	Consider the following $3\times 3$ matrix:
	\[
	A = 
	\begin{pmatrix}
	1 & 1 & 4 \\
	0 & 2 & 3 \\
	0 & 0 & 6
	\end{pmatrix}.
	\]
	The first two columns are already in signed Hermite normal form.
	When \Cref{alg:s-normal-form} encounters the third column ($j=3$), the second entry ($i=2$) is already minimal.
	At the beginning of the last iteration ($i=1$) of the inner loop, we have $G = \Z_2^3$.
	The three standard generators of $G$ act on $\Z_6$ as $x \mapsto -x+3$, $x \mapsto x+3$, and $x \mapsto -x$.
	Then $\rho(G)$ is a subgroup of $\Aff(\Z_6)$ isomorphic to $\Z_2^2$.
	On \cref{line:orbit}, we have $O = \set{1,2,4,5}$.
	By choosing $\overline S$ as the second standard generator of $\Z_2^3$, we obtain
	\[
	\varphi(\overline S) A \overline S = 
	\begin{pmatrix}
	1 & 1 & 1 \\
	0 & 2 & 3 \\
	0 & 0 & 6
	\end{pmatrix},
	\]
	which is the signed Hermite normal form.
\end{example}

\begin{proposition}
	\label{prop:complexity}
	The running time of \Cref{alg:s-normal-form} is $O(r^{\theta-1}n^2(r^2+n))$, where $\theta$ is such that two $r\times r$ integer matrices can be multiplied in time $O(r^\theta)$.
\end{proposition}

\begin{proof}
	We assume that all operations between integers require time $O(1)$.
	Then the algorithm of \cite{storjohann1996asymptotically} allows to compute the Hermite normal form $H$ of a matrix $A \in \M(r,n;\Z)$ in time $O(r^{\theta-1}n)$.
	This algorithm also returns a matrix $U$ such that $H = UA$.
	Notice that the best exponent $\theta$ is between $2$ and $2.3728639$, by \cite{le2014powers}.
	
	Throughout the execution of \Cref{alg:s-normal-form}, the group $G$ is always described by one of its $\Z_2$-bases.
	Since the $\Z_2$-dimension of $G$ increases at most by $1$ at each iteration of the outer loop (\cref{line:augment-G}), it is always bounded by $n$.
	
	The most expensive operations of \Cref{alg:s-normal-form} are the following.
	\Cref{line:phi} requires $O(n)$ computations of Hermite normal forms, and is executed $n$ times (because it is inside the outer loop), so it takes $O(r^{\theta-1} n^3)$ time.
	\Cref{line:orbit} requires to compute $O(n)$ elements of $\Aff(\Z_q)$, each of them being obtained as a dot product between two vectors of size $r$.
	It is executed $rn$ times (because it is inside the inner loop), so it takes $O(r^2n^2)$ time.
	\Cref{line:HNF-inner-loop} is executed $O(rn)$ times, and thus takes $O(r^\theta n^2)$ time.
	\Cref{line:restrict-G,line:restrict-phi} require to compute $O(n)$ multiplications of $r \times r$ matrices.
	They are executed $O(rn)$ times, so they take $O(r^{\theta+1}n^2)$ time.
	The overall running time is $O(r^{\theta-1}n^3) + O(r^2n^2) + O(r^\theta n^2) + O(r^{\theta+1} n^2) = O(r^{\theta-1}n^2(r^2+n))$.
\end{proof}

Since $\theta < 3$, assuming $r \leq n$ (otherwise some rows of $\HNF(A)$ are zero, and can be ignored), the time complexity of \Cref{alg:s-normal-form} is less than cubic in the input size $rn$.

\section{Decomposition of representable matroids}
\label{sec:decomposition}

Given an arithmetic matroid $M = (E, \rk, m)$, a \emph{decomposition} of $M$ is a partition $E = E_1 \sqcup \dots \sqcup E_k$ of the groundset such that $\rk(X) = \rk(X \cap E_1) + \dots + \rk(X \cap E_k)$ and $m(X) = m(X \cap E_1) \dotsm m(X \cap E_k)$ for every $X \subseteq E$.
An arithmetic matroid is \emph{indecomposable} if it has no non-trivial decompositions.
Notice that, if $M$ is an indecomposable arithmetic matroid, the underlying matroid $\MM = (E, \rk)$ can be decomposable.

The following lemma allows decomposing a represented arithmetic matroid into indecomposable ones, with a simple and fast algorithm.

\begin{lemma}
	Let $M = (E, \rk, m)$ be a torsion-free arithmetic matroid of rank $r$.
	Suppose that the first $r$ elements of the groundset $E$ form a basis $B$ of $M$.
	Let $A \in \M(r,n;\Z)$ be a representation of $M$, where $A$ is in Hermite normal form.
	A partition $E = E_1 \sqcup E_2$ is a decomposition of $M$ if and only if $A_{ij} = 0$ for all $(i,j) \in (B \cap E_1) \times E_2 \cup (B \cap E_2) \times E_1$.
\end{lemma}

\begin{proof}
%
	If $A_{i,j} = 0$ for all $(i,j) \in (B \cap E_1) \times E_2 \cup (B \cap E_2) \times E_1$ then, up to a permutation of rows and columns, $A$ is a block diagonal matrix having the columns in $E_1$ in the first block and the columns in $E_2$ in the second block.
	Then $E = E_1 \sqcup E_2$ is a decomposition of $M$.

	Suppose now that $E = E_1 \sqcup E_2$ is a decomposition of $M$.
	We prove the statement by induction on $j$.
	Assume without loss of generality that $j \in E_2$.

	We start with the case $j \in B$.
	We have that
	\[
		\prod_{k=1}^j A_{k,k} = m(\set{1, \dots, j}) = m(\set{1,\dots,j} \cap E_1) \cdot m(\set{1,\dots,j} \cap E_2).
	\]
	By induction, $m(\set{1,\dots,j} \cap E_1) = \prod_{k \in \set{1,\dots,j} \cap E_1} A_{k,k}$.
	In addition, we have $m(\set{1,\dotsc, j} \cap E_2) \mid \det(A')$, where $A'$ is the square submatrix of $A$ consisting of the rows $\set{i} \cup (\set{1,\dots,j-1} \cap E_2)$ and the columns $\set{1,\dots,j} \cap E_2$.
	By induction, $\det(A') = A_{i,j} \cdot \prod_{k \in \set{1,\dots,j-1} \cap E_2} A_{k,k}$.
	Putting everything together, we obtain that $A_{j,j} \mid A_{i,j}$.
	Since $A$ is in Hermite normal form, $A_{i,j} = 0$.
	
	Consider now the case $j \not\in B$.
	The $j$-th column of $A$ is a linear combination of the columns in $B \cap E_2$. Therefore $A_{i,j} = 0$.
\end{proof}

\section{Applications and examples}
\label{sec:applications}

The software library Arithmat \cite{arithmat}, which is publicly available as a Sage package, implements arithmetic matroids, toric arrangements, and some of their most important operations.
The algorithms of this paper are implemented, together with some additional ones such as Lenz's algorithm to compute the poset of layers of a toric arrangements \cite{lenz2017computing}.

As an application of our library and algorithms, we provide some examples of central toric arrangements with a non-shellable (nor Cohen-Macaulay) poset of layers, and a non-shellable (nor Cohen-Macaulay) arithmetic independence poset.
This disproves two popular conjectures in the community of arrangements and matroids.

\begin{definition}[Poset of layers]
	The poset of layers of a toric arrangement $\A$ is the set of connected components of intersections of elements of $\A$, ordered by reverse inclusion.
\end{definition}

Posets of layers of toric arrangements associated with root systems were proved to be shellable \cite{delucchi2017shellability, paolini2018shellability}, and this led to the conjecture that posets of layers are always shellable.

A related poset is the \emph{(arithmetic) independence poset} of a toric arrangement, defined in \cite[Definition 5]{lenz2017stanley}, \cite[Section 2]{martino2018face} (under the name of \emph{poset of torsions}), and \cite[Section 7]{d2018stanley} (under the name of \emph{poset of independent sets}).

\begin{definition}[Arithmetic independence poset]
	The arithmetic independence poset of a toric arrangement $\A$ is the set of pairs $(I, W)$ where $I \subseteq \A$ is an independent set and $W$ is a connected component of $\bigcap I$.
	The order relation is defined as follows: $(I_1, W_1) \leq (I_2, W_2)$ if and only if $I_1 \subseteq I_2$ and $W_1 \supseteq W_2$.
\end{definition}

D'Alì and Delucchi proved that both posets are homology Cohen-Macaulay over fields of all but a finite number of characteristics \cite{d2018stanley}.
It was conjectured that the arithmetic independence poset is shellable.
Notice that the non-arithmetic versions of these posets (the poset of flats and the independence poset of an ordinary matroid) are shellable, and therefore Cohen-Macaulay over fields of every characteristic.

Consider the example of \cite[Section 3]{Pagaria2018}: let $M$ be the arithmetic matroid associated with the matrix
\[
	A = \begin{pmatrix*}[c]
		1 & 1 & 1 & -3 \\
		0 & 5 & 0 & -5 \\
		0 & 0 & 5 & -5
	\end{pmatrix*}.
\]
Using the algorithm of \Cref{sec:representations}, we find that $M$ has $13$ non-equivalent essential representation.
These $13$ representations give rise to $3$ non-isomorphic posets of layers.
These $3$ posets are realized by the matrices $A$ and
\[
A' = \begin{pmatrix*}[c]
1 & 1 & 1 & -1 \\
0 & 5 & 0 & 5 \\
0 & 0 & 5 & -5
\end{pmatrix*}, \quad
A'' = \begin{pmatrix*}[c]
1 & 2 & 2 & 1 \\
0 & 5 & 0 & 5 \\
0 & 0 & 5 & -5
\end{pmatrix*}.
\]
The matrices $A, A', A''$ are given in signed Hermite normal form (see \Cref{sec:shnf}).
The fact that $A$ and $A'$ give rise to non-isomorphic posets of layers was already proved by the first author in \cite{Pagaria2018}.

The homology of the order complex of the poset of layers (with the bottom element removed) is equal to $(0, \Z_5, \Z^{48})$ in all $3$ cases.
In particular, these posets of layers are not Cohen-Macaulay over fields of characteristic $5$ and therefore are not shellable.

The arithmetic independence posets of the $13$ representations of $M$ are pairwise isomorphic.
Their order complexes (with the bottom element removed) have homology $(0, \Z_5, \Z^{73})$.
Therefore these posets are not Cohen-Macaulay in characteristic $5$ and are not shellable.
Our computations settle some different conjectures about the posets associated with a toric arrangement, but also highlight the following problem.

\begin{question}
	Let $M$ be an arithmetic matroid.
	Are the arithmetic independence posets of the representations of $M$ always pairwise isomorphic? 
\end{question}

\nocite{*}
\bibliographystyle{amsalpha-abbr}
\bibliography{bibliography}

\newcommand{\etalchar}[1]{$^{#1}$}
\providecommand{\bysame}{\leavevmode\hbox to3em{\hrulefill}\thinspace}
\providecommand{\MR}{\relax\ifhmode\unskip\space\fi MR }
\providecommand{\MRhref}[2]{%
  \href{http://www.ams.org/mathscinet-getitem?mr=#1}{#2}
}
\providecommand{\href}[2]{#2}
\begin{thebibliography}{BLVS{\etalchar{+}}99}

\bibitem[BLVS{\etalchar{+}}99]{OrientedMatroidsBook}
A.~Bj\"orner, M.~Las~Vergnas, B.~Sturmfels, N.~White, and G.~M. Ziegler,
  \emph{Oriented matroids}, second ed., Encyclopedia of Mathematics and its
  Applications, vol.~46, Cambridge University Press, Cambridge, 1999.
  \MR{1744046}

\bibitem[BM14]{BrandenMoci2014}
P.~Br\"and\'en and L.~Moci, \emph{The multivariate arithmetic {T}utte
  polynomial}, Transactions of the American Mathematical Society \textbf{366}
  (2014), no.~10, 5523--5540. \MR{3240933}

\bibitem[CD17]{CallegaroDelucchi2017}
F.~Callegaro and E.~Delucchi, \emph{The integer cohomology algebra of toric
  arrangements}, Advances in Mathematics \textbf{313} (2017), 746--802.
  \MR{3649237}

\bibitem[Coh93]{cohen1993course}
H.~Cohen, \emph{A course in computational algebraic number theory}, Graduate
  Texts in Mathematics, vol. 138, Springer, 1993.

\bibitem[DCP10]{deconcini2010topics}
C.~De~Concini and C.~Procesi, \emph{Topics in hyperplane arrangements,
  polytopes and box-splines}, Springer, 2010.

\bibitem[DD18]{d2018stanley}
A.~D'Al{\`\i} and E.~Delucchi, \emph{Stanley-Reisner rings for symmetric
  simplicial complexes, $G$-semimatroids and Abelian arrangements}, arXiv
  preprint 1804.07366 (2018).

\bibitem[DGP17]{delucchi2017shellability}
E.~Delucchi, N.~Girard, and G.~Paolini, \emph{Shellability of posets of labeled
  partitions and arrangements defined by root systems}, arXiv preprint
  1706.06360 (2017).

\bibitem[DM13]{DAdderioMoci2013}
M.~D'Adderio and L.~Moci, \emph{Arithmetic matroids, the {T}utte polynomial and
  toric arrangements}, Advances in Mathematics \textbf{232} (2013), 335--367.

\bibitem[{Len}17a]{lenz2017computing}
M.~{Lenz}, \emph{Computing the poset of layers of a toric arrangement}, arXiv
  preprint 1708.06646 (2017).

\bibitem[{Len}17b]{Lenz2017}
\bysame, \emph{{Representations of weakly multiplicative arithmetic matroids
  are unique}}, arXiv preprint 1704.08607 (2017).

\bibitem[{Len}17c]{lenz2017stanley}
\bysame, \emph{Stanley-Reisner rings for quasi-arithmetic matroids}, arXiv
  preprint 1709.03834 (2017).

\bibitem[LG14]{le2014powers}
F.~Le~Gall, \emph{Powers of tensors and fast matrix multiplication},
  Proceedings of the 39th International Symposium on Symbolic and Algebraic
  Computation, ACM, 2014, pp.~296--303.

\bibitem[Mar18]{martino2018face}
I.~Martino, \emph{Face module for realizable $\mathbb{Z}$-matroids},
  Contributions to Discrete Mathematics \textbf{13} (2018), no.~2, 74--87.

\bibitem[New72]{newman1972integral}
M.~Newman, \emph{Integral matrices}, Pure and Applied Mathematics, vol.~45,
  Academic Press, 1972.

\bibitem[Oxl11]{OxleyBook}
J.~Oxley, \emph{Matroid theory}, second ed., Oxford Graduate Texts in
  Mathematics, vol.~21, Oxford University Press, Oxford, 2011. \MR{2849819}

\bibitem[{Pag}17]{Pagaria2017}
R.~{Pagaria}, \emph{{Combinatorics of Toric Arrangements}}, arXiv preprint
  1710.00409 (2017).

\bibitem[{Pag}18]{PagariaOAM}
\bysame, \emph{{Orientable arithmetic matroids}}, arXiv preprint 1805.11888
  (2018).

\bibitem[{Pag}19]{Pagaria2018}
\bysame, \emph{Two examples of toric arrangements}, Journal of Combinatorial
  Theory, Series A \textbf{167} (2019), 389--402.

\bibitem[Pao18]{paolini2018shellability}
G.~Paolini, \emph{Shellability of generalized Dowling posets}, arXiv preprint
  1811.08403 (2018).

\bibitem[PP19]{arithmat}
R.~Pagaria and G.~Paolini, \emph{Arithmat: Sage implementation of arithmetic
  matroids and toric arrangements}, \url{https://github.com/giove91/arithmat},
  2019.

\bibitem[SL96]{storjohann1996asymptotically}
A.~Storjohann and G.~Labahn, \emph{Asymptotically fast computation of Hermite
  normal forms of integer matrices, ISSAC 1996}, 1996.

\end{thebibliography}

\bigskip\bigskip 

\end{document}